\documentclass[12pt]{amsart}
\usepackage[T2A]{fontenc}
\usepackage{amsmath}
\usepackage{extarrows}
\usepackage{amssymb,amsthm,amscd}
\usepackage{fullpage}
\usepackage{hyperref}
\usepackage[normalem]{ulem} 
\usepackage[toc,page]{appendix}
\usepackage{indentfirst}

\usepackage{graphicx}
\usepackage{tikz}
\usepackage{tikz-cd}
\usetikzlibrary{matrix,arrows,positioning,calc,chains}

\usepackage{enumitem} 

\def\Aut{\operatorname{Aut}}

\def\SAut{\operatorname{SAut}}

\def\id{{\mathrm{id}}}

\def\ZZ{{\mathbb Z}}

\def\NN{{\mathbb N}}
\def\G{{\mathbb G}}

\def\AA{{\mathbb A}}

\def\K{{\mathbb K}}

\def\H{{\mathcal H}}

\def\0{\circ}

\def\K{F}
\def\H{\mathcal{H}}
\def\RB{\mathcal{R}}

\theoremstyle{plain}
\newtheorem{theorem}{Theorem}[section]
\newtheorem{lemma}[theorem]{Lemma}

\newtheorem{conjecture}[theorem]{Conjecture}
\newtheorem*{conjecture*}{Conjecture}
\newtheorem{question}[theorem]{Question}
\newtheorem{proposition}[theorem]{Proposition}

\newtheorem{corollary}[theorem]{Corollary}
\theoremstyle{definition}
\newtheorem{definition}[theorem]{Definition}
\newtheorem{example}[theorem]{Example}
\newtheorem{remark}[theorem]{Remark}

\newtheorem{sit}[theorem]{}

\usepackage[draft,multiuser,english]{fixme}

\fxsetup{layout=inline}
\FXRegisterAuthor{p}{ap}{AP\!\!}

\FXRegisterAuthor{g}{vg}{VG\!\!}

\begin{document}

\author{Vsevolod Gubarev}

\address{
Sobolev Institute of Mathematics \\
Acad. Koptyug ave. 4, 630090 Novosibirsk, Russia
}

\address{
Novosibirsk State University \\
Pirogova str. 2, 630090 Novosibirsk, Russia
}

\email{wsewolod89@gmail.com}

\author{Alexander Perepechko}

\address{
  Kharkevich Institute for Information Transmission Problems\\
  19 Bolshoy Karetny per., 127994 Moscow, Russia
}

\address{
Moscow Institute of Physics and Technology (State University)\\
9 Institutskiy per., Dolgoprudny, Moscow Region, 141701, Russia 
}

\address{
National Research University Higher School of Economics\\
 20 Myasnitskaya ulitsa, Moscow 101000, Russia 
}

\email{a@perep.ru}

\sloppy

\title{Injective Rota--Baxter operators\\ of weight zero on $F[x]$}

\thanks{Vsevolod Gubarev is supported by the Program of fundamental scientific researches of the Siberian Branch of Russian Academy of Sciences, I.1.1, project 0314-2019-0001.
The research of Alexander Perepechko was supported by the grant RSF-19-11-00172.}

\begin{abstract}
Rota--Baxter operators present a natural generalisation of integration by parts formula for the integral operator.
In 2015, Zheng, Guo, and Rosenkranz conjectured that every injective Rota--Baxter operator of weight zero on the polynomial algebra $\mathbb{R}[x]$
is a~composition of the multiplication by a~nonzero polynomial and a~formal integration at some point.
We confirm this conjecture over any field of characteristic zero.
Moreover, we establish a structure of an ind-variety on the moduli space of these operators and describe an additive structure of generic modality two on it. 
Finally, we provide an infinitely transitive action on codimension one subsets.
\end{abstract}


\maketitle

\section{Introduction}

G. Baxter introduced the notion of a Rota--Baxter operator in 1960~\cite{Baxter}
as a natural generalization of  integration by parts formula for the integral operator.
Further, such operators were studied from algebraic, combinatorial, topological, physical 
and many other points of view by G.-C. Rota, P. Cartier, L. Guo and others, 
see details in~\cite{Unital,GuoMonograph}.

The integral operator is known to be injective on the algebra of continuous functions on~$\mathbb{R}$.
So, given an algebra, it is natural to classify injective Rota--Baxter operators of weight zero
on it. In~\cite{Unital}, it was proved that there are no 
injective Rota--Baxter operators of weight zero on a~unital finite-dimensional algebra.

The situation is completely different if we study Rota--Baxter operators of weight zero
on the polynomial algebra $F[x]$, where $F$~denotes the ground field of characteristic zero. 
In~\cite{Monom2}, S.~H.~Zheng, L.~Guo, and M.~Rosenkranz described all injective monomial Rota--Baxter operators of weight zero on $F[x]$; a~monomial operator is an operator which maps each monomial to a monomial with some coefficient.
More about monomial Rota--Baxter operators on polynomial algebras see in~\cite{MonomNonunital,Monom}. 
Representations of Rota--Baxter operators on $F[x]$ have been studied in~\cite{ReprRB,ReprRB2}.

A Rota--Baxter operator of weight zero on $F[x]$ is called {\it analytically modeled} if 
it equals a~composition of the multiplication $l_r$ by a fixed nonzero polynomial $r$
and a~formal integration $J_a$ at some fixed point~$a$. In~\cite{Monom2}, it was proved that up to a~constant term every injective Rota--Baxter operator on $F[x]$ acts on each monomial as an analytically modeled Rota--Baxter operator.
Based on this result, S.H.~Zheng, L.~Guo, and M.~Rosenkranz conjectured that every injective Rota--Baxter operator of weight zero on $\mathbb{R}[x]$ is analytically modeled.

Throughout the work, we operate over the ground field $F$ of characteristic zero. We confirm the conjecture of S.H.~Zheng, L.~Guo, and M.~Rosenkranz (Theorem~3.8).
This allows us to describe the moduli space $\RB$ of the injective Rota--Baxter operators of weight zero on $F[x]$ (Proposition~\ref{pr:moduli}),
which we do in terms of ind-varieties and ind-groups introduced by I.~Shafarevich in \cite{Sh}.

Furthermore, we introduce natural regular actions of the additive group $\G_a$ of the ground field. 
Given an affine variety $X$ and the group $\SAut(X)$ generated by all the additive actions on $X$, the $\SAut(X)$-action on $X$ is transitive if and only if it is infinitely transitive, see \cite{AFKKZ}.
Though $\RB$ is the quasi-affine infinite-dimensional ind-variety, we prove several results inspired by this equivalence.
Namely, the introduced actions provide an additive structure on $\RB$ of complexity two (Corollary~\ref{c:additive-Ra}), 
induce a transitive action on $\RB$ (Proposition~\ref{pr:tr}) 
and an infinitely transitive action on subsets of codimension one (Theorem~\ref{th:inf-tr}).
 
In Section 2 we state all required preliminaries.
In Section 3 we confirm the conjecture of S.H.~Zheng, L.~Guo, and M.~Rosenkranz over any field of characteristic zero.
Thus, we finish the description of all injective Rota--Baxter operators on $F[x]$ over a field of characteristic zero. 
In Section 4 we describe the structure of their moduli space $\RB$ in terms of ind-varieties, introduce families of $\G_a$-actions on~$\RB$ and describe their properties.

The authors are grateful to Ivan Arzhantsev for valuable remarks and suggestions.

\section{Preliminaries}

There are two principally different classes of Rota--Baxter operators: those of nonzero weight and those of weight zero. In this work we consider only the latter ones.

\begin{definition}
A linear operator $R$ defined on an algebra $A$
is called a \emph{Rota--Baxter operator of weight $\lambda$}, where $\lambda\in\K$, if the identity
\begin{equation}\label{RBlambda}
R(f)R(g) = R( R(f)g + fR(g) )+\lambda R(fg)
\end{equation}
holds for every $f,g\in A$.

We are interested only 
in \emph{injective Rota--Baxter operators of weight zero}. 
In particular, the identity \eqref{RBlambda} transforms into 
\begin{equation}\label{RB}
R(f)R(g) = R( R(f)g + fR(g) ).
\end{equation}
\end{definition}

\begin{example}\label{ex:non-inj}
A linear operator $R$ on $F[x]$  sending
$x^{2n} = 0$ into zero and $x^{2n+1}$ into $\frac{x^{2n+2}}{2n+2}$ for any $n\in\ZZ_{\ge0}$, is a nontrivial non-injective Rota--Baxter operator of weight zero.
\end{example}

Let us introduce the following linear operators on $F[x]$:
\begin{itemize}
    \item (multiplication) $l_r\colon f\mapsto rf$, where $r\in F[x]$; 
    \item (standard derivation) $\delta\colon x^n\mapsto nx^{n-1};$
    \item (formal integration at $a$) $J_a\colon x^n\mapsto\dfrac{x^{n+1}-a^
    {n+1}}{n+1}.$
\end{itemize}
Note that $\delta\circ J_a=\id$ and $J_a$ is a Rota--Baxter operator on $F[x]$ for any $a\in\K$.
We also denote by $I$ the operator $J_0$ for brevity.
By $\circ$ we denote a~composition of operators.

\begin{lemma}[\cite{Monom2}]\label{l:lr}
Let $r\in F[x]$ and let $R$ be a Rota--Baxter operator on $F[x]$.
Then a~linear operator $R\circ l_r$ is again a Rota--Baxter operator on $F[x]$.
\end{lemma}

In~\cite{Monom2}, S.H. Zheng, L. Guo, and M. Rosenkranz initiated a~study of injective Rota--Baxter operators on $F[x]$. They described completely all injective monomial Rota--Baxter operators of weight zero on $F[x]$; 
an operator on $F[x]$ is called \emph{monomial} if it sends each monomial to a monomial with some coefficient. For the general case, the authors made a significant advance proving

\begin{theorem}[Zheng, Guo, Rosenkranz, 2015~\cite{Monom2}]\label{th:ZGR}
Let $R$ be an injective Rota--Baxter operator on $F[x]$,
where $F$ is a field of characteristic~$0$.
Then there exists a nonzero polynomial $r\in F[x]$ such that $\delta\circ R = l_r$.
\end{theorem}

After Theorem~\ref{th:ZGR}, the following conjecture arises naturally (originally it was stated over~$\mathbb{R}$).

\begin{conjecture}[Zheng, Guo, Rosenkranz, 2015~\cite{Monom2}]\label{conj}
Every injective Rota--Baxter operator on $F[x]$
over a field $F$ of characteristic~$0$ equals $J_a \circ l_r$ for some nonzero polynomial $r\in F[x]$ and $a\in F$. 
\end{conjecture}

\section{Proof of Conjecture}

Below we express Conjecture~\ref{conj} in terms of linear functionals on~$F[x]$.
 
\begin{proposition}\label{pr:c}
Let $R$ be an operator on $F[x]$ such that $\delta\circ R=l_r$ for some nonzero polynomial $r(x)\in F[x]$.
Let us introduce a linear functional $c\colon F[x]\to F$ by the formula 
\[
c\colon f\mapsto R(f)(0).
\]
Then $R$ is a Rota--Baxter operator if and only if
\begin{equation}\label{eq:c}
    c(f)c(g) + c(I(rf)g+fI(rg)) = 0
\end{equation}
for all $f,g\in F[x]$.
Moreover, this is a one-to-one correspondence between Rota--Baxter operators $R$ on $F[x]$ satisfying $\delta\circ R=l_r$ and linear functionals $c$ on $F[x]$ satisfying~\eqref{eq:c}. Its inverse is defined by the formula $R=I\circ l_r+c$. 

\end{proposition}
\begin{proof}
Since $\delta\circ R=\delta\circ I\circ l_r$, there holds $R(f)-I(rf)\in F$ for any $f\in F[x]$, and so 
$R(f) = I(rf)+R(f)(0)$. Hence, $R = I\circ l_r + c$.

Then the condition $\eqref{RB}$ transforms into
\begin{multline*}
0 = R( R(f)g + f R(g) ) - R(f)R(g)\\
  = R\big(I(rf)g + I(rg)f + c(f)g + c(g)f\big) - (I(rf) + c(f))(I(rg) + c(g))   \\
  = I(I(rf)rg) + I(I(rg)rf) - I(rf) I(rg)
  + c(f)c(g) + c\big(I(rf)g+fI(rg)\big)\\
  = c(f)c(g) + c\big(I(rf)g+fI(rg)\big).
\end{multline*}
The last equality holds since $I\circ l_r$ is a Rota--Baxter operator by Lemma~\ref{l:lr}. 
Thus, $R$ is a~Rota--Baxter operator if and only if $c$ satisfies \eqref{eq:c}.
Since $R$ equals $I\circ l_r+c$, it is uniquely determined by $r$ and~$c$, hence the one-to-one equivalence.
\end{proof}

\begin{sit}
We denote 
\begin{enumerate}
    \item by $M_r$ the set of linear functionals satisfying \eqref{eq:c}, where $r\in F[x]$;
    \item by $c_{r,a}$ the linear functional $c_{r,a}\colon f \mapsto -I(rf)(a)$;
    \item by $N_r$ the family of linear functionals $\{c_{r,a}\mid a\in F\}$ parameterized by $a\in F$.
\end{enumerate}
\end{sit}

\begin{remark}\label{rm:Nr}
The operator $R = I\circ l_r + c_{r,a}$, corresponding to $c_{r,a}\in N_r$, equals $J_{a}\circ l_r$.
Indeed,
\[
R(f)
 = I(rf)+c_{r,a}(f)=I(rf)-I(rf)(a) 
 = J_{a}(rf).
\]
By Lemma~\ref{l:lr}, $J_{a}\circ l_r$ is a Rota--Baxter operator. Thus, $N_r\subseteq M_r$.
\end{remark}

\begin{remark}\label{rm:cn}
Let us introduce the coordinates $c_i=c(x^i)$ on the space of linear functionals and let $r=r_0+r_1x+\cdots+r_kx^k$. Then, after substituting $f=x^n,g=x^m$, the equation \eqref{eq:c} transforms into the system
\begin{equation}\label{eq:cn}
c_n c_m + \sum\limits_{i=0}^k \left(\frac{1}{i+n+1}+\frac{1}{i+m+1}\right)r_ic_{i+n+m+1} = 0, \quad n,m\in\mathbb{Z}_{\ge0}.
\end{equation}
Since $c$ is linear, it is defined by its values at 
$1,x,x^2\ldots$, so this system is equivalent to~\eqref{eq:c}.
In these coordinates $M_r\subset F^{\times\infty}=\{(c_0,c_1,\ldots)\mid c_i\in F\}$.
\end{remark}

\begin{lemma}\label{lm:P}
Let $P(c_0,\ldots,c_n)=0$ be a polynomial equation on the space of linear functionals, $P\in F[x_0,\ldots, x_n],\, n\in\NN$, such that $N_r\subset \{P(c_0,\ldots,c_n)=0\}$. Then $M_r\subset \{P(c_0,\ldots,c_n)=0\}$ as well.
\end{lemma}
\begin{proof}
The equation $P(c_0,\ldots,c_n)=0$ is equivalent on $M_r$ to some linear one $\sum\limits_{i=0}^s a_ic_i=0$ for some $s\in \ZZ_{\ge0}$, which we denote by $L$. 
In order to see this equivalence, it is enough to substitute products $c_nc_m$ repeatedly by linear combinations from corresponding equations in~\eqref{eq:cn}.

Assume that $L$ is non-trivial and the highest coefficient $a_s$ is nonzero. 
Note that $c_{r,a}(x^n)=-I(rx^n)(a)$ is a~polynomial of degree $n+k+1$ in~$a$. Thus, substituting $c_i=c_{r,a}(x^i)$, $i=0,\ldots,s$, into~$L$, we obtain a~nonzero polynomial $\sum\limits_{i=0}^s a_i c_{r,a}(x^i)$ on~$a$ of degree~$s+k+1$. 
Since $F$ is an infinite field, we get a~contradiction.
\end{proof}

Define projections 
$\pi_l\colon F^{\times\infty}\to F^{l+1}$ as follows, \[\pi_l((c_0,c_1,\ldots)) = (c_0,c_1,\ldots,c_l).\]
\begin{lemma}\label{lm:pi}
Let $r\in F[x]$ be a nonzero polynomial of degree $k=\deg(r)$. Then
\begin{enumerate}
    \item $\pi_k$ is injective on $M_r$ and $N_r$,
    \item $\pi_k(M_r)$ and $\pi_k(N_r)$ are Zariski closed subsets in $F^{k+1}$.
\end{enumerate}
\end{lemma}

\begin{proof}
We will prove both assertions at first for $M_r$ and then do the same for $N_r$.
Let us transform the system \eqref{eq:cn} as follows. For any $c_t$, where $t>k$, we may transform the equation with $m=0$ and $n=t-1-k$ to the form $c_t=P_t(c_0,\ldots,c_{t-1})$. So, $c_t$ can be expressed as a~polynomial in coordinates with lower indices. 

Any other equation corresponding to some $m,n\in\mathbb{Z}_{\ge0}$ can be reduced to the form $g_{m,n}(c_0,\ldots,c_k)$ by consecutive substitutions $c_t=P_t(c_0,\ldots,c_{t-1})$ for all $t>k$. 
Thus, the system \eqref{eq:cn} is equivalent to a system of polynomial equations
\[\{g_{m,n}(c_0,\ldots,c_k)=0\mid m>0, n\ge0\}\cup\{c_t=P_t(c_0,\ldots,c_{t-1})\mid t>k\}.\]

So, the subset $\pi_k(M_r)\subset F^{k+1}$ is defined by the polynomial system $\{g_{m,n}(c_0,\ldots,c_k)\}$, hence is Zariski closed. The map $\pi_k\colon M_r\to \pi_k(M_r)$ is a one-to-one correspondence with the inverse map $(c_0,\ldots$, $c_k)\mapsto (c_0,\ldots,c_k, P_{k+1},P_{k+2},\ldots)$. 

Since $N_r\subset M_r$, $\pi_k$ is injective on $N_r$ as well.
Consider the map $\phi\colon F\to N_r$ acting as $\phi(a)=c_{r,a}$ and  $\phi_k=\pi_k\circ\phi\colon a\mapsto (I(r)(a),I(xr)(a),\ldots,I(x^kr)(a))$. 
Then $\phi_k$ is a~morphism that maps an affine line into an affine space, hence its image $\phi_k(\mathbb{A}^1)=\pi_k(N_r)$ is Zariski closed. 

Indeed, we may extend $\phi_k$ to the morphism $\bar{\phi}_k\colon\mathbb{P}^1\to \mathbb{P}^{k+1}$ 
of a~projective line into a~projective space. Since $\mathbb{P}^1$ is a~complete algebraic variety, its image is closed. On the other hand, $\mathbb{P}^1$ does not have non-constant regular functions and thus its image cannot lie in the affine space $\mathbb{A}^{k+1}$. So, $\bar{\phi}_k(\mathbb{P}^1\setminus \mathbb{A}^1)\in \mathbb{P}^{k+1}\setminus\mathbb{A}^{k+1}$, and the closure $\phi_k(\mathbb{A}^1)$ lies in $\bar{\phi}_k(\mathbb{P}^1)\cap \mathbb{A}^{k+1}=\phi_k(\mathbb{A}^1)$.
\end{proof}

\begin{proposition}\label{pr:MN}
For any nonzero $r\in F[x]$ there holds $M_r=N_r.$
\end{proposition}
\begin{proof}
By Remark~\ref{rm:Nr}, we have $N_r\subset M_r$.
Assume that $c\in M_r\setminus N_r$. 
Then $\pi_k(c)\in \pi_k(M_r)\setminus\pi_k(N_r)$, where $k=\deg(r)$.

By Lemma~\ref{lm:pi}, $\pi_k(N_r)$ is Zariski closed, hence there exists a~polynomial $P\in F[c_0,\ldots, c_k]$ that vanishes on $\pi_k(N_r)$ but does not equal 0 at $\pi_k(c)$. 
Note that $P(b)=P(\pi_k(b))$ for any $b\in M_r$. 
Thus, $P$ vanishes on $N_r$ but not on $M_r$. 
This contradicts Lemma~\ref{lm:P}.
\end{proof}

In the following theorem we confirm Conjecture~\ref{conj}.
\begin{theorem}\label{th:conj}
Every injective Rota--Baxter operator of weight zero on $F[x]$ over a~field~$F$ of characteristic~$0$ equals $J_a \circ l_r$ for some nonzero polynomial $r\in F[x]$ and $a\in F$. 
\end{theorem}
\begin{proof}
Let $R$ be an injective Rota--Baxter operator on $F[x]$. By Theorem~\ref{th:ZGR}, there exists a~nonzero $r\in F[x]$ such that $\delta\circ R=l_r$. 

Let $c\colon f\mapsto R(f)(0)$. 
Proposition~\ref{pr:c} implies that $c\in M_r$, hence $c\in N_r$ by Proposition~\ref{pr:MN}. 
Due to Remark~\ref{rm:Nr}, the corresponding operator $R=I\circ l_r+c$ equals $J_{a}\circ l_r$ for some $a\in F$.
\end{proof}

\section{The moduli space}
Given an ascending sequence of closed embeddings of algebraic varieties \[X_1\hookrightarrow X_2\hookrightarrow \ldots,\] their inductive limit $X=\varinjlim X_n$ is called an \emph{ind-variety}.  
An \emph{morphism} $\phi\colon X\to Y$ of ind-varieties $X=\varinjlim X_n$ and $Y=\varinjlim Y_n$ is a~collection of morphisms $\phi_n\colon X_n\to Y_{n'}$ for each $n\in\NN$, where $n'\in\NN$ depends on $n$, such that $\phi_{n+1}|_{X_n}=\phi_n$. 
Two sequences $\varinjlim X_n$ and $\varinjlim X_n^\prime$ on $X$ are called \emph{equivalent}, if the identity map $\varinjlim X_n\cong\varinjlim X_n^\prime$ is an isomorphism of ind-varieties.
If all the $X_i$, $i\in\NN$, are affine (resp. quasi-affine, projective, irreducible) up to an equivalence, then $X$ is called \emph{affine} (resp. \emph{quasi-affine}, \emph{projective}, \emph{irreducible}). 

Let $X=\varinjlim X_n$ and $U=\varinjlim U_n$ be ind-varieties such that $U_n\subset X_n$ is an open (resp. closed) subset for each $n\in\NN$.
Then $U$ is called an \emph{open} (resp. \emph{closed}) ind-subvariety of $X$.
The direct product of ind-varieties $X,Y$ is defined by $X\times Y=\varinjlim X_n\times Y_n$. 

An ind-variety $G=\varinjlim X_n$ is called an \emph{ind-group} if it is endowed with a~group structure such that multiplication and inverse maps are ind-morphisms.
An action of an ind-group $G$ on an ind-variety $X$ is called \emph{regular}, if the action map $G\times X\to X$ is an ind-morphism.

Ind-varieties and ind-groups were introduced by I.~Shafarevich in 1965 \cite{Sh}, see e.g. \cite[Sections 4.1--4.3]{Kum} for precise definitions and properties and \cite{FK} for further details.

By $\G_a=\G_a(\K)$ we denote the additive group of the field~$F$. 
In this section we describe the moduli space of injective Rota--Baxter operators of weight zero on $\K[x]$ in terms of ind-varieties and additive actions on them.

\subsection*{Structure of an ind-variety}
\begin{sit}
For each $a\in\K$ and $n\in\ZZ_{\ge0}$ we denote
\begin{gather*}
\RB^a_n =\{J_a\circ l_r\mid r\in\K[x]\setminus\{0\}, \deg(r)\le n\},\\
\RB_n=\bigcup_{a\in\K}\RB^a_n,\quad
\RB^a =\bigcup_{n\in\NN}\RB^a_n,\quad
\RB=\bigcup_{a\in\K}\RB^a.
\end{gather*}
By Theorem~\ref{th:conj}, $\RB$ is the set of injective Rota--Baxter operators of weight zero on $\K[x]$.
\end{sit}

\begin{remark}
The set $\RB^a$, where $a\in\K$, together with the zero operator forms a~subspace in the space of operators.
\end{remark}

\begin{lemma}\label{l:Ja-unique}
Let $R=J_a\circ l_r\in\RB$, and let $b\in\K$ be such that $R(f)(b)=0$ for all $f\in\K[x]$.
Then $b=a$.
\end{lemma}
\begin{proof}
Up to the change of coordinates $x\mapsto x-a$ we may assume that $a=0$.
Let $r=r_0+\ldots+r_kx^k$ for some $k\in\ZZ_{\ge0}$ and substitute $f=x^j$, where $j\in\ZZ_{\ge0}$:
\[
0=R(f)(b)=I(rx^j)(b)=\sum_{i=0}^k \frac{r_ib^{i+j+1}}{i+j+1}=b^{j+1}\sum_{i=0}^k \frac{r_ib^{i}}{i+j+1}.
\]
Assume that $b\neq0.$
Then the vector $v=(r_0,r_1b,\ldots,r_kb^k)\in\K^{k+1}$ is orthogonal to each of vectors $u_j=\left(\frac{1}{j+1},\frac{1}{j+2},\ldots,\frac{1}{j+k+1}\right)\in\K^{k+1}$, where $j\in \ZZ_{\ge0}$. 
Since the vectors $u_0,\ldots,u_{k}$ form a basis of $\K^{k+1}$, we have $v=0$ and $r=0$, a~contradiction.
\end{proof}

\begin{example}
For any $d\in\NN$ the set $\K[x]_{\le d}\setminus\{0\}$ of nonzero polynomials of degree at most $d$ is an irreducible quasi-affine variety isomorphic to $\AA^{d+1}\setminus\{0\}$.
Thus, 
\[
\K[x]\setminus\{0\}=\varinjlim \K[x]_{\le d}\setminus\{0\}
\] 
is a quasi-affine ind-variety.
\end{example}

\begin{proposition}\label{pr:moduli}
 \begin{enumerate}
    \item The correspondence $\K\times(\K[x]\setminus\{0\})\to\RB$, $(a,r)\mapsto J_a\circ l_r$, is bijective.
In particular, $\RB$ admits a natural structure of a quasi-affine ind-variety.
     \item  The operator action map $\RB\times\K[x]\to\K[x]$ is an ind-morphism.
 \end{enumerate}
\end{proposition}
\begin{proof}
Let $(a,r),(a',r')\in \mathcal{R}$ be such that $J_a\circ l_r=J_{a'}\circ l_{r'}$.
Applying this operator to~1, we obtain $J_a(r)=J_{a'}(r')$. 
Discarding the constant term, we have $I(r)=I(r')$, hence $r=r'$. 
By Lemma~\ref{l:Ja-unique}, $a=a'$.

Therefore, we may identify $\RB$ with $\K\times(\K[x]\setminus\{0\})$.
For any $d\in\NN$ 
we have an irreducible quasi-affine variety
$\K\times(\K[x]_{\le d}\setminus\{0\})$ isomorphic to $\AA^1\times(\AA^{d+1}\setminus\{0\})$.
Thus, 
\[
\RB=\bigcup_{d\in\NN}\K\times(\K[x]_{\le d}\setminus\{0\})
\]
is an ind-variety. The first assertion follows.

The operator action map $\RB_n\times \K[x]_{\le n}\to \K[x]_{\le 2n+1}$ is given by the morphism
\[
((a,r_0+\ldots+r_{n}x^n), f_0+\ldots+f_{n}x^n)\mapsto \sum_{i=0}^n\sum_{j=0}^n \frac{r_if_j(x^{i+j+1}-a^{i+j+1})}{i+j+1}.
\]
The second assertion follows.
\end{proof}

\begin{definition}
Let $X$ be an ind-variety.
A subset $Z\subset \Aut(X)$ of automorphisms of $X$ together with a structure of an algebraic variety on $Z$
is called an \emph{algebraic family} of automorphisms of $X$ if 
the action map $Z\times X\to X$ is a morphism.

We say that the action of an ind-group $G$ on $X$ is \emph{universal} if for any algebraic family $Z$ of automorphisms of $X$, which is a subset of $G$, the trivial embedding $Z\to G$ is a morphism. 
\end{definition}

\begin{remark}
A universal structure of an ind-variety on $X$ is unique up to equivalence, if it exists.
In particular, for any affine algebraic variety $X$ there exists a structure of an ind-group on $\Aut(X)$ such that its action on $X$ is universal, see \cite[Theorem 5.1.1]{FK}.
\end{remark}

\begin{question}
Does the introduced structure of an ind-group on $\RB$ provide the universal action on $\K[x]$? 
\end{question}

\subsection*{Additive structure}

\begin{sit}
Given $b\in\K$ and $s\in\K[x]$ such that $s(b)=0$, we denote
\begin{align*}
h^b_s\colon \RB\to\RB,\ J_a\circ l_r\mapsto J_a\circ l_{r+r(b)s}, \\
\H^b = \{h^b_s\mid s\in\K[x], s(b)=0\}.
\end{align*}
Then $\H^b$ is a group, because $h^b_s\circ h^b_{s'}=h^b_{s+s'}$.
Further, for any $k\in\ZZ_{>0}$ we introduce the following one-parameter subgroup in $\H^b$:
\[
H^b_k =\{ h^b_{\gamma(x^k-b^k)}\mid \gamma\in \K \}.
\]
The $H^b_k$-action on $\RB$ is a $\G_a$-action.
\end{sit}

\begin{definition}
An action of an ind-group $G$ on an irreducible ind-variety $X$ is said to be of \emph{generic modality}~$k$, 
if there exists an open ind-subvariety $U\subset X$
and a morphism $U\to Z$,
where $Z$ is an algebraic variety of dimension~$k$, 
such that fibers of $\pi$ are $G$-orbits. 
\end{definition}

If $X$ is an algebraic variety, then $k=\mathrm{tr.deg.}\K(X)^G$. 
So, this definition conforms with \cite[Definition 2]{Arzh}.
The following lemma shows that this definition is correct.

\begin{lemma}
Consider the regular action of an ind-group $G$ on an irreducible ind-variety $X$.
Assume that there exist open ind-subvarieties $U,U'\subset X$ and 
morphisms to algebraic varieties $\pi\colon U\to Z$ and $\pi'\colon U'\to Z'$
such that fibers of both $\pi$ and $\pi'$ are $G$-orbits. 
Then $\dim Z=\dim Z'$. 
\end{lemma}
\begin{proof}
Since $X$ is irreducible, the intersection $W=U\cap U'$ is a~non-empty open ind-subvariety in $X$, as well as in both $U$ and $U'$.
Consider the restrictions of $\pi$ and $\pi'$ on $W$.
Their images are dense subsets $\pi(W)$ in $Z$ and $\pi'(W)$ in $Z'$ respectively. 
Let $X=\varinjlim X_n$, then $W=\varinjlim W_n$, where $W_n=X_n\cap W$ for each $n\in\NN$.
So, for some $n\in\NN$ the images $\pi(W_n)$ and $\pi'(W_n)$ are dense in $Z$ and $Z'$ respectively.
Since the generic fibers of $\pi|_{W_n}$ and $\pi'|_{W_n}$ have the same dimension, the assertion follows.
\end{proof}

\begin{definition}
By an \emph{additive structure of generic modality} $k$ on an ind-variety $X$ we mean an action of an abelian unipotent ind-group $G$ on $X$ of generic modality $k$.
\end{definition}

\begin{proposition}\label{pr:Ha}
Let $a,b\in\K$.
\begin{enumerate}[label=(\roman*)]
\item The group $\H^b$ is an abelian unipotent ind-group. Namely,
\[
\H^b 
 = \varinjlim_k  (H^b_1\times\ldots \times H^b_k) \cong \varinjlim_k (\G_a)^k.
\]
\item Consider the evaluation map $\pi_b\colon\RB^a\to \K[x]/(b)\cong\K$, $J_a\circ l_r \mapsto r(b)$. Then $\pi_b^{-1}(c)$ is an $\H^b$-orbit for any $c\neq0$ and $\pi_b^{-1}(c)\cap\RB_k$ is an $H_1^b\times\ldots\times H_k^b$-orbit.
Thus, $\H^b$ (resp. $H^b_k$) acts on $\RB^a$ (resp. $\RB^a_k$) with generic modality one.
\end{enumerate}
\end{proposition}
\begin{proof}
For any $k\in\NN, c_1,\ldots,c_k\in\K$, $r\in\K[x]$, there holds
$$
\big(H_1^b(c_1)\circ\ldots\circ H_k^b(c_k)\big)(J_a\circ l_r) = J_a\circ l_{r'},
$$
where 
$r' = r+r(b)\sum\limits_{i=1}^k c_i(x^i-b^i)$.
The commutativity follows. Thus, 
$ H_1^b\times\ldots\times H_k^b\cong (\G_a)^k$, which is known to be unipotent. So, we have~(i). 

Consider the map $\phi^k\colon\RB^a_k\to\K$, $J_a\circ l_r\mapsto r(b)$.
Each fiber of this map, except the preimage of~0, is an 
$\big(H_1^b\times\ldots\times H_k^b\big)$-orbit.
The maps $\{\phi^k\mid k\in\ZZ_{>0}\}$ provide the ind-morphism $\RB^a\to\K$, a~generic fiber of which is an $\H^b$-orbit, hence (ii).
\end{proof}

\begin{corollary}\label{c:additive-Ra}
The ind-variety $\RB$ admits an additive structure of generic modality two.
\end{corollary}
\begin{proof}
Let us fix $b\in\K$ and consider an open ind-subvariety 
\[
U=\{J_a\circ l_r\mid a\in\K,\, r\in\K[x],\, r(b)\neq0\}
\]
in $\RB$.
The fibers of the map 
\[
\pi\colon U\to \K\times\K^\times,\, J_a\circ l_r\mapsto (a,r(b))
\]
are $\H^b$-orbits. The assertion follows.
\end{proof}
\subsection*{Infinite transitivity}
Here we establish a transitive action on ordered tuples of operators in $\RB^a$.
\begin{proposition}\label{pr:m-tr}
The group $\H=\langle \H^b\mid b\in\K\rangle$ acts transitively on ordered $m$-tuples of linearly independent operators in $\RB^a$ for any $m\in\NN$.
\end{proposition}

\begin{sit}
For any pairwise distinct $b_1,\ldots,b_m\in\K $ and arbitrary $c_1,\ldots$, $c_m\in\K^\times$ we define a set
$$
A(b_1,\ldots,b_m\mid c_1,\ldots,c_m)
 = \left\{
(J_a\circ r_1,\ldots,J_a\circ r_m)\mid
r_i\in\K[x]\setminus\{0\},\ 
r_i(b_j)= \delta_{ij}c_i\right\},
$$
where $\delta_{ij} = \begin{cases} 1, & i=j, \\
0, & i\neq j. \end{cases}$
\end{sit}

\begin{lemma}\label{l:m-tr-1}
Any $m$-tuple of linearly independent operators in $\RB^a$ can be sent to some $A(b_1,\ldots,b_m\mid c_1,\ldots,c_m)$ by an element of $\H$.
\end{lemma}
\begin{proof}
Consider such an $m$-tuple $(R_1,\ldots,R_m)$
and let $R_i=J_a\circ l_{r_i}$ for some $r_i\in\K[x]$, where $i=1,\ldots,m$. 
Then $r_1,\ldots,r_m$ are linearly independent.
There exist $b_1,\ldots,b_m$ such that $\det((r_i(b_j)))\neq0$.

Given $k,l\in\{1,\ldots,m\}$ such that $k\neq l$, let us denote $s_l=\prod_{j\neq l}\frac{x-b_j}{b_l-b_j}$ and consider the element $h^{b_k}_{\lambda s_l}\in\H^{b_k}$. 
Applying it to $R_1,\ldots,R_m$, we obtain an elementary transformation of the matrix $(r_i(b_j))$: the $k$th column is added to the $l$th one with the coefficient $\lambda$.
Using such transformations for suitable $k,l$, and $\lambda$, we may leave just one nonzero value at each row of $(r_i(b_j))$.
The assertion follows.
\end{proof}
\begin{lemma}\label{l:m-tr-2}
The group $\H$ acts on $A(b_1,\ldots,b_m\mid c_1,\ldots,c_m)$ transitively.
\end{lemma}
\begin{proof}
Let $(R_1,\ldots,R_m),(R_1^\prime,\ldots,R_m^\prime)\in A(b_1,\ldots,b_m\mid c_1,\ldots,c_m).$
We have $R_1=J_a\circ l_r$ and $R_1^\prime=J_a\circ l_{r'}$ for some $r,r'\in\K[x]$ such that $r(b_1)=r'(b_1)=c_1$.
Then 
$$
r=q\prod_{i=2}^m(x-b_i),\quad r'=q'\prod_{i=2}^m(x-b_i)
$$ 
for some $q,q'\in\K[x]$. 
Thus, $h^{b_1}_s$ for $s=\frac{r'-r}{c_1}$ sends $(R_1,R_2,\ldots,R_m)$ to $(R_1^\prime,R_2,\ldots,R_m)$. 

Acting in a similar way successively for $R_2,\ldots, R_m$, we may send $(R_1,\ldots,R_m)$ to $(R_1^\prime,\ldots,R_m^\prime)$.
\end{proof}

\begin{lemma}\label{l:m-tr-3}
The sets $A(b_1,\ldots,b_m\mid c_1,\ldots,c_m)$ and $A(b_1^\prime,\ldots,b_m^\prime\mid c_1^\prime,\ldots,c_m^\prime)$ have nonempty intersection if the sets $\{b_1,\ldots,b_m\}$ and $\{b_1^\prime,\ldots,b_m^\prime\}$ are disjoint.
\end{lemma}
\begin{proof}
It is enough to take $r_k$ such that $r_k(b_i)=r_k(b_i^\prime)=0$ for $i\neq k$ and $r_k(b_k)=c_k$, $r_k(b_k^\prime)=c_k^\prime$ for each $k=1,\ldots,m$.
Then the tuple $(J_a\circ l_{r_1},\ldots,J_a\circ l_{r_m})$ belongs to both $A(b_1,\ldots,b_m\mid c_1,\ldots,c_m)$ and $A(b_1^\prime,\ldots,b_m^\prime\mid c_1^\prime,\ldots,c_m^\prime)$.
\end{proof}
\begin{proof}[Proof of Proposition~\ref{pr:m-tr}]
Let $R=(R_1,\ldots,R_m)$ be an $m$-tuple of linearly independent operators in $\RB^b$. 
By Lemma~\ref{l:m-tr-1}, we may send it to some tuple $R'$ in $A_1=A(b_1,\ldots,b_m\mid c_1,\ldots,c_m)$. 
Let us take a set $\{b_1^\prime,\ldots,b_m^\prime\}$ disjoint with both $\{b_1,\ldots,b_m\}$ and $\{1,\ldots,m\}$. 
By Lemmas \ref{l:m-tr-2} and \ref{l:m-tr-3}, we may send $R'$ to some tuple $R''\in A(b_1^\prime,\ldots,b_m^\prime\mid 1,\ldots,1)$ and then to a tuple $R^{(3)}\in A(1,\ldots,m\mid 1,\ldots,1)$. 
Thus, any tuple can be sent to the set $A(1,\ldots,m\mid 1,\ldots,1)$, which enjoys the transitive $\H$-action.
Proposition is proved.
\end{proof}

\begin{remark}\label{rm:Ha-lin}
The action of the group $\H$ is linear.
In particular, if a~nontrivial linear combination of operators $R_1,\ldots,R_m$ in $\RB^a$ is zero, then the same linear combination of $h.R_1,\ldots,h.R_m$ is zero for any element $h\in\H$ as well.
\end{remark}

\begin{definition}
A group $G$ is said to act on a set $S$ \emph{infinitely transitively} if it acts transitively on the set of ordered $m$-tuples of pairwise distinct points in $S$ for any $m\in\NN$.
\end{definition}

\begin{sit}
For each $a\in\K$ and $s\in\K[x]$ such that $s(a)=0$ we introduce the linear operator
\[
h^{b,2}_s\colon \RB\to\RB,\; J_a\circ l_r\mapsto J_a\circ l_{r+r(b)^2s}.
\]
\end{sit}

\begin{theorem}\label{th:inf-tr}
The group generated by $\G_a$-actions $\big\{h^{b,2}_s,h^b_s\mid b\in\K,\,s\in\K[x],\,s(b)=0\big\}$ acts on $\RB^a$ infinitely transitively for any $a\in\K$.
\end{theorem}
\begin{proof}
Let us prove by induction by $m$ that for any $m\in\NN$ the introduced group acts transitively on $m$-tuples of pairwise distinct operators in $\RB$. 
If $m=1$, then the assertion follows from Proposition~\ref{pr:m-tr}. 
Assume that the assertion holds for $(m-1)$-tuples of operators. 

Using Proposition~\ref{pr:m-tr}, it is enough to prove that any $m$-tuple $(R_1,\ldots,R_m)$ of pairwise distinct operators can be sent to an $m$-tuple of linearly independent ones. 
By induction, we may send $R_1,\ldots$, $R_{m-1}$ to $J_a\circ l_1,J_a\circ
l_x,\ldots,J_a\circ l_{x^{m-2}}$ respectively. So, we assume that $R_i=J_a\circ
l_{x^{i-1}}$ for $i=1,\ldots,m-1$.
If $R_m$ is linearly independent with $R_1,\ldots,R_{m-1}$, we are done by Proposition~\ref{pr:m-tr}.
Otherwise, we have $R_m=J_a\circ l_r$ for some $r=\sum_{i=0}^{m-2} r_ix^i\in\K[x]_{\le m-2}$.

The group $\H^0$ fixes $R_2,\ldots,R_{m-1}$. 
The element $h^{0,2}_{x^m}$ sends $R_1$ and $R_m$ to $J_a\circ l_{1+x^m}$ and $J_a\circ l_{r+r_0^2x^m}$ respectively. If $r_0\notin\{0,1\}$, then the images of $R_1,\ldots,R_m$ are linearly independent. 
If $r_0=0$, then by the inductive hypothesis
we may send $R_2,\ldots,R_m$ into any $(m-1)$-tuple of linearly independent operators. 
This is an open condition, hence we may keep $R_1$ linearly independent with $R_2,\ldots,R_m$ as well.
Then we are done.
Thus, we assume that $r_0=1$.

By Proposition~\ref{pr:m-tr}, we can permute operators $R_1,\ldots,R_{m-1}$. 
By Remark~\ref{rm:Ha-lin}, coefficients $r_0,\ldots,r_{m-2}$ are permuted accordingly. 
Hence we may repeat the argument above for each coefficient and assume that $r_0=r_1=\ldots=r_{m-2}=1$, so $r=1+x+\ldots+x^{m-2}$. 
Then the element $h^{1,2}_{x^m}$ sends $R_1,\ldots,R_m$ into a tuple of linear independent operators. Indeed, for the map 
\[
\varphi\colon \K[x]\to\K[x],\quad f\mapsto f+f(1)^2(x^m-1),
\] 
we have
\begin{align*}
\varphi(x^i) & = x^i + (x^m-1),\ i\in\{0,\ldots,m-2\},\\
\varphi(r) & = 1+x+\ldots+x^{m-2} + (m-1)^2(x^m-1).
\end{align*}
Thus, the images $\varphi(1),\varphi(x),\ldots,\varphi(x^{m-2}),\varphi(r)$
are linearly independent.
\end{proof}

\subsection*{Action of $\Aut(\K[x])$ and transitivity}
Here we describe the induced action of $\Aut(\K[x])$ on $\RB$.
This allows us to present a collection of $\G_a$-actions on $\RB$ 
such that the group generated by them acts transitively on $\RB$.
\begin{sit}
Consider the action of $\Aut(\K[x])$ on $\RB$ by conjugation.
Since
\[
\Aut(\K[x])=\{x\mapsto \mu x+\nu\mid \mu\in\K^\times, \nu\in\K\}\cong \G_a\rtimes\G_m,
\]
this provides a $\G_a$- and $\G_m$-action on $\RB$, which we denote by $G_a$ and $G_m$ respectively. More formally, we have
\begin{align*}
G_a\colon& \G_a\times\RB\to\RB,
\quad (\nu,J_{a}\circ l_{r(x)})\mapsto 
J_{a-\nu}\circ l_{r(x+\nu)},\\
G_m\colon& \G_m\times\RB\to\RB,
\quad(\mu,J_{a}\circ l_{r(x)})\mapsto 
J_{\frac{a}{\mu}}\circ l_{r(\mu x)}.
\end{align*}
\end{sit}

\begin{remark}
The $\Aut(\K[x])$-orbits on $\RB$ are parameterized by polynomials in $\K[x]$, which leading coefficient equals one.
Given such a~polynomial $r=x^n+\ldots\in\K[x]$, the corresponding orbit is 
\[
\{J_a\circ l_{r(\mu(x-a))}\mid a\in\K, \mu\in\K^\times\}.
\]
\end{remark}

\begin{proposition}
The action map $\Aut(\K[x])\times \RB\to\RB$ is an ind-morphism.
Moreover, an element $g\in\Aut(\K[x])$ sends each subset $\RB^a$, where $a\in\K$, to some subset $\RB^{a'}$, where $a'\in\K$. 
\end{proposition}
\begin{proof}
It is easy to check that $g$ restricted to $\RB^k$, where $k\in\ZZ_{\ge0}$, is an automorphism of a~quasi-affine variety. The first statement follows.

Assume that $g$ sends an operator $J_a\circ l_r$ to  $J_{a'}\circ l_{r'}$ for some nonzero $r,r'\in\K[x]$. Then $a'$ does not depend on $r$, and the second statement follows.
\end{proof}

\begin{proposition}\label{pr:tr}
The group $\langle G_a,\H\rangle$ acts transitively on $\RB$.
\end{proposition}
\begin{proof}
We have to prove that an element $R_1=J_{a_1}\circ l_{r_1}$ can be sent to $R_2=J_{a_2}\circ l_{r_2}$ for any $a_1,a_2\in\K$ and $r_1,r_2\in\K[x]\setminus\{0\}$. 
Since $G_a(a_1-a_2)(R_1)\in \RB^{a_2}$, we may assume that $a_1=a_2$. 

Since $r_1$ and $r_2$ are nonzero and the field $\K$ is infinite, there exist $a,a'\in\K$ which are not roots of the polynomials $r_1,r_2$.
Then for some $\gamma\in\K$ there holds $r_1(a')+\gamma r_1(a)(a'-a)=r_2(a').$ 
Thus, by Proposition~\ref{pr:Ha}, $h^a_\gamma(R_1)$ and $R_2$ belong to the same orbit of $\H^{a'}$. The statement follows.
\end{proof}


\begin{thebibliography}{67}
\bibitem{Arzh}
I.V. Arzhantsev, \textit{On modality and complexity of affine embeddings}. SB MATH (8) {\bf 192} (2001), 1133--1138.

\bibitem{AFKKZ}
I.~Arzhantsev, H.~Flenner, S.~Kaliman, F.~Kutzschebauch, and M.~Zaidenberg,
\textit{Flexible varieties and automorphism groups}. Duke Math. J. (4) {\bf 162} (2013), 767--823.

\bibitem{Baxter}
G. Baxter,
\textit{An analytic problem whose solution follows from a simple algebraic identity}. Pacific J. Math. {\bf 10} (1960), 731--742.

\bibitem{FK} 
J.-P.~Furter, H.~Kraft, 
\textit{On the geometry of the automorphism groups of affine varieties}, arXiv:1809.04175, 179~p.

\bibitem{Unital}
V. Gubarev,
\textit{Rota--Baxter operators on unital algebras}. Mosc. Math.~J. (accepted), arXiv.1805.00723v3, 43~p.

\bibitem{MonomNonunital}
V. Gubarev,
\textit{Monomial Rota--Baxter operators on free commutative non-unital algebra}. Siberian Electron. Math. Rep. {\bf 17} (2020), 1052--1063.

\bibitem{GuoMonograph}
L. Guo,
\textit{An Introduction to Rota--Baxter Algebra}. Surveys of Modern Mathematics, vol. 4, Intern. Press, Somerville (MA, USA); Higher Education Press, Beijing, 2012.

\bibitem{Kum} 
S. Kumar, 
\textit{Kac-Moody groups, their flag varieties and representation theory}.
Progress in Mathematics, vol. 204, Birkh\"auser Boston Inc., Boston, MA, 2002.

\bibitem{ReprRB} 
L.~Qiao, J.~Pei,
\textit{Representations of polynomial Rota--Baxter algebras}.
J. Pure Appl. Algebra (7) {\bf 222} (2018), 1738--1757.

\bibitem{Sh}
I. R. Shafarevich, 
\textit{On some infinite-dimensional groups}. Rend. Mat. e Appl. (5) {\bf 25} (1966), 208--212.

\bibitem{ReprRB2}
X. Tang, 
\textit{Modules of polynomial Rota-Baxter algebras and matrix equations}, arXiv:2003.05630, 16~p.

\bibitem{Monom}
H. Yu,
\textit{Classification of monomial Rota--Baxter operators on $k[x]$}. J.~Algebra Appl. {\bf 15} (2016), 1650087, 16~p.

\bibitem{Monom2}
S.H. Zheng, L. Guo, and M. Rosenkranz,
\textit{Rota--Baxter operators on the polynomial algebras, integration and averaging operators}. Pacific J. Math. (2) {\bf 275} (2015), 481--507.
\end{thebibliography}
\end{document}